\documentclass{amsart}
\usepackage{amsfonts,amssymb,amscd,amsmath,enumerate,verbatim,calc}
\usepackage[all]{xy}

\newcommand{\CM}{Cohen-Macaulay}

\newcommand{\wrt}{with respect to}
\newcommand{\Sq}{\mathcal{S}_q }

\newcommand{\db}{\mathbf{d}_{d,0} }

\newcommand{\q}{\mathfrak{q} }
\newcommand{\A}{\mathfrak{a} }
\newcommand{\Pc}{\mathcal{P} }
\newcommand{\Qc}{\mathcal{Q} }

\newcommand{\F}{\mathbb{F} }
\newcommand{\xb}{\mathbf{x} }
\newcommand{\G}{\mathbb{G}^\bullet }
\newcommand{\rt}{\rightarrow}

\newcommand{\depth}{\operatorname{depth}}
\newcommand{\Modg}{\operatorname{ ^* Mod}}
\newcommand{\modg}{\operatorname{ ^* mod}}
\newcommand{\Ar}{\operatorname{ ^* A}}
\newcommand{\ann}{\operatorname{ann}}
\newcommand{\codim}{\operatorname{codim}}

\newcommand{\Sym}{\operatorname{Sym}}

\newcommand{\Ext}{\operatorname{Ext}}
\newcommand{\gHom}{\operatorname{ ^* Hom}}

\theoremstyle{plain}

\newtheorem{theorem}{Theorem}[section]
\newtheorem{corollary}[theorem]{Corollary}

\newtheorem{proposition}[theorem]{Proposition}

\theoremstyle{definition}

\newtheorem{s}[theorem]{}
\newtheorem{remark}[theorem]{Remark}

\theoremstyle{remark}

\begin{document}

\title[Annihilators]{Top dickson class and annihilators of cohomology over invariant rings}
\author{Tony~J.~Puthenpurakal}
\date{\today}
\address{Department of Mathematics, IIT Bombay, Powai, Mumbai 400 076, India}

\email{tputhen@gmail.com}

 \begin{abstract}
Let $\mathbb{F}_q$ denote the finite field with $q = p^r$ elements. Let $V$ be a finite dimensional vector space
of dimension $d$ over $\mathbb{F}_q$ and let $G \subseteq GL(V)$ be a  group. Let $R = \mathbb{F}_q[V] = \text{Sym}(V^*)$ and let $S = R^G$.
 Let $\mathbf{d}_{d,0} $ be the top Dickson class, i.e., $\mathbf{d}_{d,0} = \prod_{0\neq v \in V^*}v$. Surprisingly  (a power of) $\db$ annihilates many cohomological modules.

  (a) Let $H^i(G, R)$ be the $i^{th}$-group cohomology of $R$ considered as a $S$-module. Set $J_i = \ann_S H^i(G, R)$.  We show that $\mathbf{d}_{d,0}  \in \sqrt{J_i}$ for all $i \geq 1$.

(b) We also show that $\mathbf{d}_{d,0}  \in \sqrt{ \ann_S H^j_{S_+}(S)}$ for all $ 0 \leq j \leq d - 1$ (here $H^j_{S_+}(S)$ is the $j^{th}$ local cohomology of $S$ with respect to $S_+$).

 As an application we get that there exists a fixed power of $\mathbf{d}_{d,0} $ which works as a cohomological annihilator.
\end{abstract}
 \maketitle
\section{introduction}
\s \label{setup}\emph{Setup:} Let $\F_q$ denote the finite field with $q = p^r$ elements. Let $V$ be a finite dimensional vector space
of dimension $d$ over $\F_q$ and let $G \subseteq GL(V)$ be a  group. Let $R = \F_q[V] = \Sym(V^*)$ and let $S = R^G$.   Let $\db =  \prod_{0\neq v \in V^*}v$ be the top Dickson class. Note $\db \in S$.  In this paper we give abundant cases having a (power of) $\db$  as a cohomological annhilator.

(1) Let $H^i(G, R)$ be the $i^{th}$-group cohomology of $R$ considered as a $S$-module. Let $J_i = \ann_S H^i(G, R)$. By Hilbert's Theorem 90 we get that $J_i \neq  0$ for $i \geq 1$, see \ref{90}. We first show
\begin{theorem}
\label{main}(with hypotheses as in \ref{setup}) For $i \geq 1$ we have $\db \in \sqrt{J_i}$.
 \end{theorem}

(2) \emph{Annihilators of local cohomology modules :} (with hypotheses as in \ref{setup}). Let $S_+ = \bigoplus_{n \geq 1}S_n$ be the irrelevant maximal ideal of $S$. Let $H^i_{S_+}(S)$ be the $i^{th}$-local cohomology modue of $S$ \wrt \ $S_+$. We show
\begin{theorem}\label{loc}(with hypotheses as in \ref{setup}). For $0 \leq j \leq d - 1$ we have
$$ \db \in \sqrt{ \ann_S H^j_{S_+}(S)}.$$
\end{theorem}

\emph{Applications:} Annhilators of Local cohomology give a lot of information on annhilators of cohomology in many constructs. For $0 \leq j \leq d - 1$, let $a_j \geq 1$ be the smallest power of $\db$ such that $\db^{a_j}\in \ann_S H^j_{S_+}(S)$. Set $\q_i = \prod_{j = 0}^{i} \db^{a_j}$ for $0 \leq j \leq d -1$. Our first application is a precise version of a result of a result of Roberts, see \cite[Theorem 1]{R} (also see \cite[8.1.2]{BH}).
\begin{corollary}
\label{rob} (with hypotheses as in \ref{setup}). Let $\G = 0 \rt G^0 \rt G^1 \rt \ldots \rt G^r \rt 0$ be a complex of free graded $S$-modules and homogeneous maps such that $H^i(\G)$ has finite length for all $i$. Then $\q_i \in \ann_S H^i(\G)$ for $0 \leq i \leq d-1$.
\end{corollary}
Our next application is a precise version of a result due to Schenzel, see \cite[Theorem 2]{Sc}  (also see \cite[8.1.4]{BH}).
\begin{corollary}
\label{Sch}(with hypotheses as in \ref{setup}). Let $\xb = x_1,\dots, x_d$ be a homogeneous system of parameters of $S$. Then
$$ \q_{d-1} \in \ann_S \frac{(x_1, \ldots x_{t-1})\colon x_t}{(x_1, \ldots, x_{t-1})} \quad \text{for t = 1, \ldots, d}.$$
\end{corollary}
Finally we give a precise version of a result in \cite[8.1.3]{BH}.
\begin{corollary}
\label{Sch-bh}(with hypotheses as in \ref{setup}).
Let $\xb = x_1,\dots, x_m$ be a homogeneous set of elements in $S$ such that $\codim (\xb) = m$. Then for $i = 1, \ldots, m$ we get that $\q_{d-i}$ is in the annhilator of the Koszul homology $H_i(\xb)$.
\end{corollary}

 \s \emph{Technique to prove our results:}
 For $i \geq 0$
let $\Sq^i$ for $p = 2$ and $\Pc^i$ (for $p \geq 3$) denote the Steenrod operators on $R$ and on $S$. An ideal $I$ of ($R$ or $S$) is said to be $\Pc^*$-invariant if $\Pc^i(I) \subseteq I$ for all $i \geq 1$ (analogously for $\Sq^*$-invariant ideals).  It is known that if $I$ is a non-zero $\Pc^*$-invariant ideal of $S$ then $\db \in \sqrt{I}$, see \cite[11.4.4]{S}.
We show (see section three) that there exists $\F_q$-linear maps $\Qc^m \colon H^i(G, R) \rt H^i(G, R)$ for all $m \geq 0 $ with $\Qc^0$ being the identity;  such that for any $s \in S$ and $\alpha \in H^i(G, R)$ we have
\[
\Qc^m(s\alpha) = \sum_{a + b = m}\Pc^a(s)\Qc^b(\alpha).
\]
 Using this fact it is not difficult to show that $J_i = \ann_S H^i(G, R)$ is a $\Pc^*$-invariant ideal in $S$, see \ref{invar}. By Hilbert's Theorem 90 we get that $J_i \neq  0$ for $i \geq 1$, see \ref{90}. Using these results  we get a proof of Theorem \ref{main}.

 \begin{remark}
Our operators $\Qc^m \colon H^i(G, R) \rt H^i(G, R)$ should \emph{not} be confused with the Steenrod operators on $H^*(G, \F_p)$ (see \cite[4.4.4 and 4.4.5]{B}).
\end{remark}

 Theorem \ref{loc} is an application of the Ellingsrud-Skjelbred spectral sequences \cite{ES}. The applications follows from standard results as given in \cite[Chapter 8]{BH}.

 We note that for \emph{any} group $G \subseteq GL_d(\F_q)$ the ring of invariants contains $\db$ as $\db$ is an invariant of $GL_d(\F_q)$. It is suprising to us that the (power of) same element annihilates cohomology (only the powers might be the different in each case.

We now describe in brief the contents of this paper. In section two we discuss a few preliminary results that we need. In section three we construct our operators $\Qc^m$ on $H^*(G, R)$.
In section four we prove Theorem \ref{main}. Finally in section five we prove Theorem \ref{loc}.

\section{Preliminaries}
In this section we discuss some preliminary results that we need.

\textbf{I}: \emph{Group cohomology}.

\s\label{setup-gc-intro}  Let $S$ be a commutative Noetherian ring.
 Let $G$ be a finite group. Let $S[G]$ be the group ring and let $Mod(S[G])$ be the category of left $S[G]$-modules. Let $(-)^G$ be the functor of $G$-fixed points. Let $H^n(G,-)$ be the $n^{th}$ right derived functor of
$(-)^G$. Let $M$ be an $S[G]$-module. We call $H^n(G, M)$ the \emph{$n^{th}$-cohomology group} of $G$ (with coefficients in $M$). Note this cohomology module also depends on $S$ which we have suppressed. We note that
$H^n(G, M) = \Ext^n_{S[G]}(S, M)$.

Group cohomology behaves well with respect to localization. This fact is well-known, for instance see \cite[2.3]{P}.
\begin{proposition}
\label{localization}(with hypotheses as in \ref{setup-gc-intro})  Let $M$ be a $S[G]$-module.  Let $W$ be a multiplicatively closed subset of $S$. Then $W^{-1}M$ is a $W^{-1}S[G]$-module and $W^{-1}H^i(G, M) \cong H^i(G, W^{-1}M)$ for all $i \geq 0$.
\end{proposition}

As a consequence we have
\begin{proposition}
\label{90}(with hypothesis as in \ref{setup}) For all $i \geq 1$ we have\\ $J_i = \ann_S H^i(G, R)$ is a non-zero ideal of $S$.
\end{proposition}
\begin{proof}
  Let $W = S \setminus \{ 0 \}$. Then notice $W^{-1}R = K$ the quotient field of $R$. So for $i \geq 1$ we have  $W^{-1}H^i(G, R) = H^i(G, K)$, see \ref{localization}. We note that $K$ is a Galois extension of the field $W^{-1}S = L$  with Galois group $G$, (see \cite[ Chapter VI, 1.8]{L}). By the additive form of Hilbert Theorem 90, see
  \cite[6.3.7]{W},  $H^i(G, K)  = 0$ for all $i \geq 1$. The result follows.
\end{proof}

\s \label{res}(with hypotheses as in \ref{setup-gc-intro}).  To prove our results we need a more concrete definition of group cohomology. Let $M$ be a $S[G]$-module. For $n \geq 0$ let
$C^n(G, M)$ be the $S$-module of all functions from $G^n$ to $M$; (here $G^0 = \{1\}$). We define the differential $d^n \colon C^n(G, M) \rt C^{n+1}(G, M)$ by
\begin{align*}
  (d^n(\psi))(g_1, \ldots, g_{n+1}) =  g_1\psi(g_2, \ldots, g_{n+1}) + &\sum_{i = 1}^{n}\psi(g_1, \ldots, g_ig_{i+1}, \cdots, g_{n+1})  \\
  & + (-1)^{n+1}\psi(g_1, \ldots, g_n).
\end{align*}
Clearly $d^n$ is $S$-linear. Also $d^{n+1}\circ d^n = 0$. So $C^*(G, M)$ is a complex. Its cohomology is $H^*(G, M)$, see \cite[6.5.5]{W}.

\textbf{II:} \emph{Ellingsrud-Skjelbred spectral sequences.}

In this sub-section we describe the Ellingsrud-Skjelbred spectral sequences \cite{ES} (we follow the exposition given in \cite[8.6]{Lorenz}).
\s Let $S$ be a commutative Noetherian ring. Let $Mod(A)$ be the category of left $A$-modules.

(1) Let $G$ be a finite group. Let $S[G]$ be the group ring and let $Mod(S[G])$ be the category of left $S[G]$-modules.

(2) Let $\A$ be an ideal in $S$. Let $\Gamma_\A(-)$ be the torsion functor associated to $\A$.  Let $H^n_\A(-)$ be the $n^{th}$ right derived functor of
$\Gamma_\A(-)$. Usually $H^n_\A(-)$ is called the $n^{th}$ local cohomology functor of $A$ \wrt \ $\A$.

(3) If $M \in Mod(S[G])$ then note $\Gamma_\A(M) \in Mod(S[G])$.

\s Ellingsrud-Skjelbred spectral sequences  are constructed as follows: Consider the following sequence of functors
\[
(i) \quad   Mod(S[G]) \xrightarrow{(-)^G} Mod(S) \xrightarrow{\Gamma_\A} Mod(S),
\]
\[
(ii)   \quad   Mod(S[G]) \xrightarrow{\Gamma_\A} Mod(S[G]) \xrightarrow{(-)^G} Mod(S)
\]
We then notice
\begin{enumerate}[\rm (a)]
\item
The above compositions are equal.
\item
It is possible to use Grothendieck spectral sequence of composite of functors to both (i) and (ii) above; see
\cite[8.6.2]{Lorenz}.
\end{enumerate}
Following Ellingsrud-Skjelbred  we let $H^n_\A(G,-)$ denote the $n^{th}$ right derived functor of this composite
functor. So by (i) and (ii) we have two first quadrant spectral sequences for each $S[G]$-module $M$
\[
(\alpha)\colon \quad \quad  E_2^{p,q} = H^p_\A(H^q(G, M)) \Longrightarrow H^{p+q}_\A(G,M), \ \text{and}
\]
\[
(\beta)\colon \quad \quad \mathcal{E}_2^{p,q} = H^p(G, H^q_\A(M)) \Longrightarrow H^{p+q}_\A(G,M).
\]
\begin{remark}\label{grading}
(1) If $S$ is $\mathbb{N}$-graded ring then $S[G]$ is also $\mathbb{N}$-graded (with $\deg \sigma = 0$ for all $\sigma \in G$).

(2)  If $M$ is a finitely generated graded left $S[G]$-module then $H^n(G, M)$ are finitely generated graded $S$-module. This can be easily seen by taking a graded free resolution of $S$ consisting of finitely generated free  graded $S[G]$-modules.

(3) Ellingsrud-Skjelbred spectral sequences have an obvious graded analogue.
\end{remark}

\textbf{III:} \emph{Steenrod operators}.

We follow the exposition given in \cite{S}.
\s Let $p$ be a prime, $q = p^s$ and let $V$ be a $\F_q$-vector space of dimension $d$. Define
Define $P(\xi)\colon \F_q[V] \rt \F_q[V][\xi]$ (with $\deg \xi = 1-q$) by the rules:
\begin{enumerate}
  \item $P(\xi)(v)  = v + v^q\xi$  for all $v \in  V^*$.
  \item $P(\xi)(u + w) = P(\xi)(u) + P(\xi)(w)$ for all $u,w \in F_q[V]$.
  \item $P(\xi)(uw) = P(\xi)(u)P(\xi)(w)$ for all $u,w \in F_q[V]$.
  \item $P(\xi)(1) = 1$.
   \end{enumerate}
  We note that $P(\xi)$ is a ring homomorphism of degree zero.

By separating out homogeneous components we obtain $\F_q$ linear maps \\ $\Pc^i \colon \F_q[V] \rt \F_q[V]$ by the requirement
$$P(\xi)(f) = \sum_{i \geq 0}\Pc^i(f)\xi^i.$$
Note $\Pc^0$ is the identity.
The operations $\Pc^i$ are called the \emph{Steenrod reduced power operations} over $\F_q$ and when $p = 2$, also denoted by $\Sq^i$ and refereed to as \emph{Steenrod squaring operations}.
We will use the following property:
$$\Pc^k(uv) = \sum_{i+j = k}\Pc^i(u)\Pc^j(v). $$

\s Let $G \subseteq GL(V)$ be a finite group. Let $S = \F_q[V]^G$. We define an action of $G$ on $\F_q[V][\xi]$ by giving the trivial action on $\xi$. We note that $\sigma(P(\xi)(r)) = P(\xi)(\sigma(r))$ for all $\sigma \in G$. Thus the maps $\Pc^i$ restrict to maps from $S$ to $S$.

\s An ideal $I$ of $S$ is said to be \emph{$\Pc^*$-invariant} if $\Pc^i(I)\subseteq I$ for all $i \geq 0$.

\textbf{IV:} \emph{Annhilators}

In this subsection $T = \bigoplus_{n \geq 0}T_n$ is a Noetherian graded ring with $T_0 = K$ a field (we \emph{do not} assume $T$ is standard graded). Let $\Modg(T)$ denote the category of all graded $T$-modules and let  $\modg(T)$ denote the category of all finitely generated graded $T$-modules. Let $\Ar(T)$ denote the category of all graded $^*$-Artinian modules. Note if $M \in \Modg(T)$ then $\ann_T M$ is a homogeneous ideal of $T$.

Let $E = E_T(K)$ be the injective hull of $K = T/T_+$. Note that $E$ is a graded $T$-module. If $M \in \Modg(T)$ set $M^\vee = \gHom(M, T)$.

\s (\emph{Matlis duality:}) (see \cite[3.6.17]{BH}). If $X \in \modg(T)$ then $X^\vee \in \Ar(T)$. If $Y \in \Ar(T)$ then $Y^\vee \in \modg(T)$. Furthermore $X^{\vee \vee}\cong X$ and $Y^{\vee \vee} \cong Y$.

The following result is well-known and easy to prove.
\begin{proposition}
\label{ann-dual} Let $X \in \modg(T)$ and $Y \in \Ar(T)$. Then
$$\ann_T X^\vee = \ann_T X \quad \text{and} \quad \ann_T Y^\vee = \ann_T Y.$$
\end{proposition}

\s Recall $D \in \Modg(T)$ is said to be a \emph{sub-quotient} of $M \in \Modg(T)$ if there exits graded submodule $U, V$ of $M$ with $U \subseteq V$ and $V/U \cong D$.

We will need the following easily proved result:
\begin{proposition}\label{ann}(with setup as above) We have
\begin{enumerate}[\rm (1)]
  \item If
 $D \in \Modg(T)$ is a sub-quotient of $M \in \Modg(T)$ then $\ann_T M \subseteq \ann_T D$.
  \item Let $X_i \in \modg(T)$ for $1 \leq i \leq 3$. If we have an exact sequence of graded modules $X_1 \rt X_2 \rt X_3$ then
  \[
  \sqrt{\ann_T X_1} \cap \sqrt{\ann_T X_3} \subseteq \sqrt{\ann_T X_2}.
  \]
  \item Let $Y_i \in \Ar(T)$ for $1 \leq i \leq 3$.  If we have an exact sequence of graded modules $Y_1 \rt Y_2 \rt Y_3$ then
  \[
  \sqrt{\ann_T Y_1} \cap \sqrt{\ann_T Y_3} \subseteq \sqrt{\ann_T Y_2}.
  \]\qed
\end{enumerate}
\end{proposition}

\s \label{bella}(with hypotheses as above). Let $T_+ = \bigoplus_{n \geq 1}T_n$. Let $H^i_{T_+}(-)$ be the $i^{th}$-local cohomology (\wrt \ $T_+$) functor. Let $M \in \modg(T)$. Then
\begin{enumerate}[\rm (1)]
  \item $H^i_{T_+}(M) \in \Ar(T)$ for all $i \geq 0$.
  \item $\ann_T M \subseteq \ann_T H^i_{T_+}(M)$ for all $i \geq 0$.
\end{enumerate}
\section{Construction of our operators}
In this section we prove
\begin{theorem}\label{op}
(with hypotheses as in \ref{setup}) Fix $i \geq 0$. There exists $\F_q$-linear maps \\ $\Qc^m \colon H^i(G, R) \rt  H^i(G, R)$ for $m \geq 0$ with $\Qc^0$ being the identity such that for any $s \in S$ and $\alpha \in H^i(G, R)$ we have
\[
\Qc^m(s\alpha) = \sum_{a + b = m}\Pc^a(s)\Qc^b(\alpha).
\]
\end{theorem}
\begin{proof}
Let $\Pc^m \colon R \rt R$ be the Steenrod operators on $R$.
We work with the chain complex $C^*(G, R)$, see \ref{res}. For $\psi \in C^n(G, R)$ define $\Qc^m(\psi) = \Pc^m\circ \psi$. Clearly $\Qc^m \colon C^*(G, R) \rt C^*(G, R)$ are $\F_q$-linear maps.
Also as $\Pc^0$ is the identity we obtain that $\Qc^0$ is the identity.
If we show that it is  a chain map then we have $\F_q$-linear maps $\Qc^m \colon H^i(G, R) \rt H^i(G, R)$.
To do this let $\psi \in C^n(G, R)$. Then
\begin{align*}
  (d^n(\Qc^m(\psi))(g_1, \ldots, g_{n+1}) &=  g_1\Pc^m(\psi(g_2, \ldots, g_{n+1})) + \\
   &\sum_{i = 1}^{n}\Pc^m(\psi(g_1, \ldots, g_ig_{i+1}, \cdots, g_{n+1}))  \\
  & + (-1)^{n+1}\Pc^m(\psi(g_1, \ldots, g_n)).
\end{align*}
We also have
\begin{align*}
  \Qc^m(d^n((\psi))(g_1, \ldots, g_{n+1}) &= \Pc^m( g_1(\psi(g_2, \ldots, g_{n+1}))) + \\
   &\sum_{i = 1}^{n}\Pc^m(\psi(g_1, \ldots, g_ig_{i+1}, \cdots, g_{n+1}))  \\
  & + (-1)^{n+1}\Pc^m(\psi(g_1, \ldots, g_n)).
\end{align*}
Notice
$$ g_1\Pc^m(\psi(g_2, \ldots, g_{n+1})) = \Pc^m( g_1(\psi(g_2, \ldots, g_{n+1}))).$$
Thus $\Qc^m \colon C^*(G, R) \rt C^*(G, R)$ are $F_q$-linear chain maps and so induce $F_q$-linear maps $\Qc^m \colon H^i(G, R) \rt H^i(G, R)$.

Let $\alpha = [\psi] \in H^i(G, R)$ and let $s \in S$. We have
\begin{align*}
  \Qc^m(s\psi)(g_1, \ldots, g_i) &=  \Pc^m(s\psi(g_1, \ldots, g_m)) \\
   &= \sum_{a + b = m}\Pc^a(s)\Pc^b(\psi(g_1, \ldots, g_i)). \\
  &= \sum_{a + b = m}\Pc^a(s)\Qc^b(\psi)((g_1, \ldots, g_i)).
\end{align*}
The result follows.
\end{proof}

\begin{remark}
  It is not difficult to show that the operators $\Qc^m$ define a module structure of $H^i(G, R)$ over the Steenrod algebra. However we have been unable to find an application of this fact.
\end{remark}
\section{Proof of Theorem \ref{main}}
\s (with hypotheses as in \ref{setup}). Recall $J_i = \ann_S H^i(G, R)$ for $i \geq 0$. We first show
\begin{proposition}\label{invar}
  The ideals $J_i$ are $\Pc^*$-invariant ideals of $S$.
\end{proposition}
\begin{proof}
  We note that $J_0 = 0$. So we have nothing to prove when $i = 0$. So assume $i \geq 1$. Let $t \in J_i$ and let $\alpha \in H^i(G, R)$. We prove by induction on $m$ that $\Pc^m(t) \in J_i$. We note $\Pc^0 = $ identity. So assume $m \geq 1$. When $m = 1$ we have
  \[
  0 = \Qc^1(t\alpha) = \Pc^1(t)\alpha + t\Qc^1(\alpha) = \Pc^1(t)\alpha.
  \]
  The second equality holds as $\Qc^1(\alpha) \in H^i(G, R)$.

  Now assume that $m \geq 2$ and $\Pc^r(t) \in J_i$ for $r < m$.
  By \ref{op} we get
  \begin{align*}
    0 = \Qc^m(t\alpha) & = \sum_{a + b = m}\Pc^a(t)\Qc^b(\alpha) \\
    &= \Pc^m(t)\alpha + \sum_{a+b = m, \ \text{and} \   a < m}\Pc^a(t)\Qc^b(\alpha) \\
     &= \Pc^m(t)\alpha.
  \end{align*}
  The result follows.
\end{proof}
We now give
\begin{proof}[Proof of Theorem \ref{main}]
Fix $i \geq 1$. Also $J_i = \ann_S H^i(G, R)$. By Proposition \ref{90} we get that $J_i \neq 0$. By
Proposition \ref{invar} we get that $J_i$ is a $\Pc^*$-invariant ideal of $S$.  It follows that $\db \in \sqrt{J_i}$, see \cite[11.4.4]{S}.
\end{proof}

\section{Proof of Theorem \ref{loc}}
In this section we give
\begin{proof}[Proof of Theorem \ref{loc}]
We note that if $d \leq 3$ then $S$ is \CM \ (see \cite{S-3}) and so we have nothing to prove. Also if $p$ does not divide order of $G$ then $S$ is \CM, and so we have nothing to prove. So we assume that $p$ divides $|G|$.  If $d \geq 4$ then by a result of Ellingsrud-Skjelbred \cite[3.9]{ES} we have
\[
\depth S \geq \min\{ \dim_{\F_q}V^P + 2, d \},
\]
where $P$ is a Sylow $p$-subgroup of $G$. We note that $\dim_{\F_q} (V)^P \geq 1$, see \cite[8.2.1]{S}. So $\depth S \geq 3$. Thus $H^j_{S_+}(S) = 0$ for $j \leq 2$.

We use the Ellingsrud-Skjelbred  spectral sequence with $\A = S_+$ and $M = R$. We have two first quadrant spectral sequences
\[
(\alpha)\colon \quad \quad  E_2^{p,q} = H^p_{S_+}(H^q(G, R)) \Longrightarrow H^{p+q}_{S_+}(G,R), \ \text{and}
\]
\[
(\beta)\colon \quad \quad \mathcal{E}_2^{p,q} = H^p(G, H^q_{S_+}(R)) \Longrightarrow H^{p+q}_{S_+}(G,R).
\]
We first analyse the spectral sequence $(\beta)$. We note that as $R$ is a finite $S$-module we have $H^q_{S_+}(R) = H^q_{R_+}(R) = 0$ for $q \neq d$. So the spectral sequence $(\beta)$ collapses. In particular $H^n_{S_+}(G, R) = 0$ for $n < d$.

Next we analyse the spectral sequence $(\alpha)$.
We note that for $q > 0$ we get by Theorem \ref{main} and \ref{bella}  that
\begin{equation*}
  \db \in \sqrt{\ann_S E_2^{p,q}}  \quad \text{for $q > 0$}. \tag{*}
\end{equation*}
For $r \geq 3$ we have that $E^{p,q}_r$ is a sub-quotient of $E^{p,q}_{r-1}$. So   we get that
\begin{equation*}
\db \in \sqrt{\ann_S E_r^{p,q}} \quad \text{for $q > 0$ and $r \geq 3$}. \tag{**}
\end{equation*}

For the convenience of the reader we first give a detailed proof of \\ $\db \in \sqrt{\ann_S H^3_{S_+}(S)}$ when $d \geq 4$.
We have $E_2^{3,0} = H^3_{S_+}(S)$.
Note $E^{3,0}_4 = E^{3,0}_\infty$. As $E^{3,0}_\infty$ is a sub-quotient  of $H^3_{S_+}(G, R) = 0$ we get that $E^{3,0}_\infty = 0$.
So we have a surjective map
$$ E_3^{0,2} \rt E_3^{3,0}  \rt 0.$$
By (**) and \ref{ann} we get that $\db \in \sqrt{\ann_S E_3^{3,0}}$.
We have an exact    sequence
$$E_2^{1, 1} \rt E_2^{3,0}  \rt E_3^{3,0} \rt 0. $$
By (*), \ref{ann} and the fact that $\db \in \sqrt{\ann_S E_3^{3,0}}$ we get that
$\db \in \sqrt{\ann_S E_2^{3,0}}$. The result follows.

Now let $3 \leq j \leq d -1$.  We have $E^{j,0}_2 = H^j_{S_+}(S)$.
Note $E^{j,0}_{j+1} = E^{j,0}_\infty$. As $E^{j,0}_\infty$ is a sub-quotient  of $H^j_{S_+}(G, R) = 0$ we get that $E^{j,0}_\infty = 0$.
So we have a surjective map
$$ U_j  \rt E_j^{j,0}  \rt 0,$$
where $U_j$ is $E_j^{p,q}$ for some $p,q$ with $q >0$.
By (**) and \ref{ann} we get that $\db \in \sqrt{\ann_S E_j^{j,0}}$.
Next we have an exact sequence
$$U_{j-1} \rt E_{j-1}^{j,0} \rt E_j^{j,0} \rt 0.$$
By (**), \ref{ann} and the fact that $\db \in \sqrt{\ann_S E_j^{j,0}}$ we get that
$\db \in \sqrt{\ann_S E_{j-1}^{j,0}}$.
Iterating we obtain $\db \in \sqrt{\ann_S E_{2}^{j,0}}$. The result follows.
\end{proof}
\begin{remark}
We did not specify the precise $p,q$ such that $U_i = E_i^{p,q}$ other than the fact that $q > 0$ since it is not germane to our discussion.
\end{remark}


\begin{thebibliography} {99}

\bibitem{B}
D.~J.~Benson,
\emph{Representations and cohomology. II},
Cambridge Stud. Adv. Math., 31
Cambridge University Press, Cambridge, 1998.

\bibitem {BH}  W. Bruns and J. Herzog,
Cohen-Macaulay Rings, revised edition,
Cambridge Studies in Advanced Mathematics, 39.
Cambridge University Press, 1998.


\bibitem{ES}
 G.~Ellingsrud, T.~Skjelbred,
 \emph{Profondeur d'anneaux d'invariants en caract\'{e}ristique p} (French),
  Compositio Math. 41 (1980), no. 2, 233--244.



\bibitem{L}
S.~Lang,
\emph{Algebra},
 Revised third edition. Graduate Texts in Mathematics, 211. Springer-Verlag, New York, 2002.

\bibitem{Lorenz}
 M.~Lorenz,
 \emph{Multiplicative invariant theory},
  Encyclopaedia of Mathematical Sciences, 135. Invariant Theory and Algebraic Transformation Groups, VI. Springer-Verlag, Berlin, 2005.

\bibitem{P}
T.~J.~Puthenpurakal,
\emph{The Cohen-Macaulay Property of Invariant Rings Over Ring of Integers of a Global Field},
 Transformation Groups (2024). https://doi.org/10.1007/s00031-024-09891-y

\bibitem{R}
P.~Roberts,
\emph{Two applications of dualizing complexes over local rings},
Ann. Sci. École Norm. Sup. (4) 9 (1976), no. 1, 103--106.

\bibitem{Sc}
P.~Schenzel,
\emph{Dualizing complexes and systems of parameters},
J. Algebra 58 (1979), no. 2, 495--501.

 \bibitem{S}
  L.~Smith,
\emph{Polynomial invariants of finite groups},
Res. Notes Math., 6
A K Peters, Ltd., Wellesley, MA, 1995.

\bibitem{S-3}
\bysame,
\emph{Some rings of invariants that are Cohen-Macaulay},
Canad. Math. Bull. 39 (1996), no. 2, 238--240.





  \bibitem{W}
  C.~A.~Weibel,
  \emph{An introduction to homological algebra},
Cambridge Stud. Adv. Math., 38
Cambridge University Press, Cambridge, 1994.



\end{thebibliography}
\end{document}